\documentclass[a4paper, 12pt]{amsart}

\usepackage{amssymb}
\usepackage{amsthm}
\usepackage{amsmath}
\usepackage{amscd}

\usepackage{comment}

\usepackage[mathscr]{eucal}
\usepackage{enumitem}

\theoremstyle{plain}
\newtheorem{thm}{Theorem}[section]
\newtheorem{prop}[thm]{Proposition}
\newtheorem{cor}[thm]{Corollary}
\newtheorem{lem}[thm]{Lemma}

\theoremstyle{definition}
\newtheorem{df}{Definition}[section]

\theoremstyle{remark}

\newtheorem*{ac}{Acknowledgements}

\newcommand{\zz}{\mathbb{Z}}
\newcommand{\qq}{\mathbb{Q}}
\newcommand{\rr}{\mathbb{R}}
\newcommand{\cc}{\mathbb{C}}

\newcommand{\yoemph}[1]{\emph{#1}}

\newcommand{\yourysp}{\mathbb{U}}
\newcommand{\yourydis}{\rho}

\DeclareMathOperator{\card}{Card}

\newcommand{\supp}{\text{{\rm supp}}}

\newcommand{\yosub}{\subseteq}

\DeclareMathOperator{\yodiam}{diam}

\newcommand{\yocball}{\mathrm{B}}

\newcommand{\yochara}{\mathscr}

\newcommand{\youfin}{\yochara{N}}

\newcommand{\yofin}{\yochara{F}}

\newcommand{\yoofam}[1]{\mathbf{TEN}(#1)}

\newcommand{\yorsub}[1]{R_{#1}}

\newcommand{\yopetal}[2]{\Pi(#1, #2)}

\newcommand{\yomapsco}[2]{\mathrm{G}(#1, #2)}
\newcommand{\yomaindisco}{\mathord{\vartriangle}}

\newcommand{\yogf}[1]{\mathbb{F}_{#1}}

\newcommand{\yocoef}[2]{\boldsymbol{C}(#1, #2)}

\newcommand{\yorcf}[2]{\mathfrak{K}(#1, #2)}
\newcommand{\yorcff}{\zeta}

\newcommand{\yovring}[2]{\mathfrak{A}(#1, #2)}
\newcommand{\yovideal}[2]{\mathfrak{o}(#1, #2)}

\newcommand{\youhaq}[4]{\mathbb{A}_{#3, #4}(#1, #2)}
\newcommand{\youhavq}[4]{U_{#1, #2,  #3, #4}}

\newcommand{\youit}{\tau}

\newcommand{\yoha}[2]{\mathbb{H}(#1, #2)}
\newcommand{\yohav}[2]{v_{#1, #2}}

\newcommand{\yopideal}[2]{\mathbb{N}_{#1, #2}}
\newcommand{\yopf}[3]{\mathbb{P}_{#3}(#1, #2)}
\newcommand{\yopfv}[3]{V_{#1, #2, #3}}

\newcommand{\youulc}[4]{\mathbb{B}_{#3, #4}(#1, #2)}

\newcommand{\yolc}[2]{\mathbb{L}[#1, #2]}

\newcommand{\yopd}[2]{\mathbb{D}[#1, #2]}
\newcommand{\yoplc}[3]{\mathbb{M}_{#3}[#1, #2]}

\newcommand{\yostan}[3]{\mathrm{St}_{#1, #2, #3}}

\newcommand{\yoproj}{\mathrm{Pr}}

\newcommand{\yowitto}[1]{\mathbb{W}(#1)}
\newcommand{\yowitt}[1]{\mathrm{Fr}\mathbb{W}(#1)}
\newcommand{\yowiv}[1]{w_{#1}}

\newcommand{\yorep}{J}

\newcommand{\yonabs}[3]{\lVert #3\rVert_{#1, #2}}
\newcommand{\yonval}[2]{v_{#1, #2}}

\newcommand{\yoratio}{\eta}

\newcommand{\yosmf}[1]{\boldsymbol{#1}}

\newcommand{\yopnbf}[1]{\qq_{#1}}
\newcommand{\yopnbcf}[1]{\cc_{#1}}
\newcommand{\yopnbv}[1]{v_{#1}}
\newcommand{\yopnburf}[1]{\widehat{\yopnbf{#1}^{\mathrm{unr}}}}

\newcommand{\yolaurentf}[1]{#1((t))}
\newcommand{\yolaurentcf}[1]{\widehat{\overline{\yolaurentf{#1}}}}

\newcommand{\yogset}{\mathcal{G}}
\newcommand{\yoqpset}{\mathcal{CH}}
\newcommand{\yowusp}[1]{\mathbb{V}_{#1}}
\newcommand{\yowudis}[1]{\sigma_{#1}}

\makeatletter
\@addtoreset{equation}{section}

\makeatother
\begin{document}

%title
\title[Algebraic structures on Urysohn spaces]
{
Algebraic structures on 
non-Archimedean Urysohn universal 
metric spaces
}

%author
\author[Yoshito Ishiki]
{Yoshito Ishiki}

%adress
\address[Yoshito Ishiki]
{\endgraf
Department of Mathematical Sciences
\endgraf
Tokyo Metropolitan University
\endgraf
Minami-osawa Hachioji Tokyo 192-0397
\endgraf
Japan
}

%email
\email{yoshito.ishiki@riken.jp}

\date{\today}

%abstract
\begin{abstract}
We investigate valued-field structures on Urysohn universal ultrametric spaces. We introduce $p$-adic Levi--Civita fields as subfields of $p$-adic Hahn fields and treat them together with ordinary Levi--Civita fields. For a subgroup $G$ of $\mathbb{R}$ containing $\mathbb{Z}$ and a countably infinite perfect field $k$, the corresponding Levi--Civita valued field is isometric to the $R$-Urysohn universal ultrametric space, where $R=\{0\}\cup\{\eta^{-g}\mid g\in G\}$. Thus these spaces admit field structures extending prescribed prime valued fields, including $\mathbb{Q}$ with the trivial valuation and the $p$-adic fields $\mathbb{Q}_{p}$. We also prove that complete valued fields with infinite residue fields are haloed, and hence universal for separable ultrametric spaces with corresponding distance sets. In the separable case, such a valued field is itself isometric to the corresponding Urysohn space. Examples include $\mathbb{C}_{p}$, the completion of the maximal unramified extension of $\mathbb{Q}_{p}$, Laurent series fields, and completions of their algebraic closures. Finally, for a countably infinite perfect residue field, the corresponding full Hahn-type valued field is a Urysohn universal ultrametric space exactly when its value group is order-isomorphic to $\mathbb{Z}$. 
\end{abstract}

%%%%%%%%%%%%%%%%%%%%%%%%%%
%%%%%%%%%%%%%%%%%%%%%%%%%%
%%%%%%%%%%%%%%%%%%%%%%%%%%
%%%%%%%%%%%%%%%%%%%%%%%%%%
%%%%%%%%%%%%%%%%%%%%%%%%%%
%%%%%%%%%%%%%%%%%%%%%%%%%%
%%%%%%%%%%%%%%%%%%%%%%%%%%
%%%%%%%%%%%%%%%%%%%%%%%%%%
%%%%%%%%%%%%%%%%%%%%%%%%%%
%%%%%%%%%%%%%%%%%%%%%%%%%%

%%%%%%%%%%%%%%%%%%%%%%%%%%%%%%%%%%%%%%%%%%%
%Start!
\maketitle
%%%%%%%%%%%%%%%%%%%%%%%%%%%%%%%%%%%%%%%%%%

\section{Introduction}\label{sec:intro}
%\subsection{Backgrounds}

The Urysohn universal metric space, 
introduced by Urysohn 
\cite{Ury1927}, 
is characterized by 
universality and strong homogeneity. 
Its ultrametric counterpart is 
the Urysohn universal ultrametric space. 
Such spaces are closely related to 
homogeneous ultrametric spaces and to 
the general structure theory of ultrametric spaces; 
see 
\cite{delhomme2015homogeneous, 
MR2754373, 
MR2667917}. 
In the non-separable setting, 
petaloid ultrametric spaces play the role of 
Urysohn universal ultrametric spaces 
for arbitrary range sets 
\cite{Ishiki2023const, 
Ishiki2023Ury, 
Ishiki2023halo}. 

Another line of research asks 
whether universal metric spaces can carry 
algebraic operations compatible with 
their metrics. 
Cameron and Vershik investigated 
isometry groups and compatible group structures 
on the Urysohn space 
\cite{MR2258622}. 
Niemiec realized 
Urysohn universal spaces as 
metric groups of exponent
$2$
and 
developed a theory of 
universal valued Abelian groups 
\cite{MR2507687, 
MR3010064}. 
Doucha constructed 
non-abelian and abelian group structures 
on universal metric spaces 
\cite{MR3294611, 
MR3646786}. 
Motivated by these results, 
we investigate the analogous problem for 
non-Archimedean valued fields: 
we ask when Urysohn universal ultrametric spaces 
can themselves be realized as valued fields. 
This problem is closely related to 
the non-Archimedean Arens--Eells 
isometric embedding theorem for valued fields 
\cite{ishiki2026nonarchimedeanarenseellsisometricembedding}, 
which embeds arbitrary non-Archimedean metric spaces into 
spherically complete Hahn-type valued fields. 
The present paper instead seeks valued-field structures 
on the Urysohn universal ultrametric spaces themselves. 

Our main tools are 
Hahn fields, 
$p$-adic
Hahn fields, 
and their Levi--Civita subfields. 
Hahn-type fields are classical objects in 
valuation theory, 
and maximally complete fields of this kind 
go back to Kaplansky's work 
\cite{MR0006161}. 
For
$p$-adic
Hahn fields, 
we use Poonen's construction 
\cite{MR1225257}; 
for the algebraic structure of 
Hahn--Mal'cev--Neumann rings, 
see also 
\cite{Poonen2026units}. 
Lampert studied a related approach to algebraic 
$p$-adic
expansions using ordinal series 
\cite{Lampert1986algebraic}. 
Ordinary Levi--Civita fields are 
well-known non-Archimedean valued fields; 
see, for instance, 
\cite{MR3782290, 
MR4367483}. 
We introduce 
$p$-adic
Levi--Civita fields 
as suitable subfields of 
$p$-adic
Hahn fields. 
We then show that ordinary and 
$p$-adic
 Levi--Civita fields provide 
valued-field models of 
Urysohn universal ultrametric spaces, 
including their non-separable generalizations. 

As a consequence, 
we obtain field structures on 
Urysohn universal ultrametric spaces 
that extend prescribed prime valued fields. 
More precisely, 
if 
$p$
is either
 $0$
or a prime, 
then every Urysohn universal ultrametric space 
arising from our construction 
admits a field structure extending 
$\yopnbf{p}$,
where 
$\yopnbf{0}=\qq$
 is equipped with 
the trivial valuation 
(see Proposition
\ref{prop:primevaluedextension}
and Theorem 
\ref{thm:mainury}). 
We also establish a universality result for 
valued fields with sufficiently large residue fields 
(see Theorem 
\ref{thm:fin} and 
Corollary
 \ref{cor:padicc}).

This paper is organized as follows. 
In Section \ref{sec:pre}, 
we review the required preliminaries on 
metric and ultrametric spaces, 
valued rings, 
and valued fields. 
We also introduce 
Hahn fields, 
$p$-adic
 Hahn fields, 
and the corresponding 
Levi--Civita fields, 
and establish the algebraic facts used later. 
In Section 
\ref{sec:Urysohn}, 
we recall Urysohn universal ultrametric spaces 
and petaloid ultrametric spaces, 
realize them using ordinary and 
$p$-adic
Levi--Civita fields, 
and prove our field-structure theorem. 
In Section 
\ref{sec:univvalued}, 
we establish  results on the universality for complete valued fields 
and compare full Hahn fields with 
Urysohn universal ultrametric spaces.

\begin{ac}

The author would like to thank 
Tomoki Yuji for helpful 
advice on algebraic arguments.

The author wishes to express his 
deepest gratitude 
to 
all members of 
Photonics Control Technology Team (PCTT) in
 RIKEN, 
where
the majority of this paper was written, 
for their invaluable support. 
Special thanks are extended to 
the Principal Investigator of PCTT,
Satoshi Wada, for 
his encouragement and support, which 
transcended disciplinary boundaries.

This work was partially supported by JSPS 
KAKENHI Grant Number 
JP24KJ0182. 
\end{ac}

\section{Preliminaries}\label{sec:pre}

Some parts of this section overlap with the preliminary material in 
\cite{ishiki2026nonarchimedeanarenseellsisometricembedding}. 
We recall them here in order to make the present paper self-contained 
and to fix the notation used below. 

\subsection{Generalities}\label{subsec:gen}
\subsubsection{Metric spaces}
For a metric space 
$(X, d)$,
for a subset 
$A$
of 
$X$,
and 
for
$x\in X$,
we define 
$d(x, A)=d(A, x)=\inf\{\, d(a, x)\mid a\in A\, \}$.
For  
$a\in X$
and 
$r\in (0, \infty)$,
we denote by 
$B(a, r; d)$
the closed ball 
centered at
 $a$
 with radius 
$r$.
We often simply represent it 
as
 $B(a, r)$
when 
no confusion can arise. 
Similarly, we define the 
open ball 
$U(a, r)$.

\subsubsection{Valued rings}

Let 
$A$
be a commutative ring. 
We say that a function 
$v\colon A\to \rr\sqcup \{\infty\}$
is 
a
\emph{(additive) valuation}
if the following conditions are satisfied:
%enubegin
\begin{enumerate}

\item 
for every 
$x\in A$,
we have 
$v(x)=\infty$
if and only if 
$x=0$;

\item 
for every pair 
$x, y\in A$,
we have 
$v(xy)=v(x)+v(y)$;

\item 
for every pair 
$x, y\in A$,
we have 
$v(x+y)\ge v(x)\land v(y)$,
where 
$\land$
stands for the 
minimum operator on 
$\rr$.
\end{enumerate}
%enuend
If 
$A$
is a field, 
then
 $(A, v)$
  is called a 
\emph{valued field}. 
Note that for every valued ring 
$(A, v)$,
we can extend the valuation 
$v$
to the field of fractions 
$K$
of 
$A$.
Namely, for 
$x=b/a\in K$,
where 
$a, b\in A$
and 
$a\neq 0$,
we define 
$v(x)=v(b)-v(a)$.
In this case, 
$v$
is well-defined and 
the pair 
$(K, v)$
naturally becomes 
a valued field. 
For example, 
for a prime 
$p$,
the 
$p$-adic
valuation 
$\yopnbv{p}$
 on
  $\zz$
  is defined by 
$\yopnbv{p}(0)=\infty$
and 
$\yopnbv{p}(x)=n$
 when 
$x\neq 0$
and 
$p^{n}$
 is the highest power of 
 $p$
  dividing
   $x$.
The completion 
$\zz_{p}$
of
 $\zz$
with respect 
to 
$\yopnbv{p}$
is called 
the 
\emph{ring of
$p$-adic
integers}. 
The field of fractions
$\yopnbf{p}$
 of
  $\zz_{p}$
is called 
the 
\emph{field of 
$p$-adic
numbers}. 
For more information on 
$p$-adic
 numbers, 
we refer the reader to 
\cite{MR0554237} 
and 
\cite{MR2444734}.

We say that a function 
$\lVert *\rVert \colon A\to [0, \infty)$
is 
a
\emph{(non-Archimedean) absolute value}
or 
\emph{multiplicative valuation}
if the following conditions are satisfied:
%enustart
\begin{enumerate}

\item 
for every 
$x\in A$,
we have 
$\lVert x \rVert=0$
if and only if 
$x=0$;

\item 
for every pair 
$x, y\in A$,
we have 
$\lVert xy\rVert =\lVert x\rVert \cdot \lVert y\rVert$;

\item 
for every pair 
$x, y\in A$,
we have 
$\lVert x+y\rVert \le  \lVert x\rVert \lor  \lVert y\rVert$,
where 
$\lor$
stands for the 
maximum operator on 
$\rr$.
\end{enumerate}
%endenu

In what follows, for every 
$\yoratio\in (1, \infty)$,
we adopt the conventions 
$\yoratio^{-\infty}=0$
and 
$-\log_{\yoratio}(0)=\infty$.
For a valuation 
$v$
on a ring 
$A$,
and for a real  number 
$\yoratio\in (1, \infty)$,
we define 
$\yonabs{v}{\yoratio}{x}=\yoratio^{-v(x)}$
and 
$\yonval{\lVert *\rVert}{\yoratio}(x)
=-\log_{\yoratio}(\lVert x\rVert)$.

For a fixed 
$\yoratio\in (1, \infty)$,
valuations and non-Archimedean absolute values on a ring 
$A$
correspond as follows.
%propstart
\begin{prop}\label{prop:rho}
Let 
$A$
be a commutative ring. 
Then the following statements are true: 
%enustart
\begin{enumerate}[label=\textup{(\arabic*)}]

\item 
For every 
$\yoratio\in (1, \infty)$,
and 
for every valuation 
$v$
on 
$A$,
the function 
$\yonabs{v}{\yoratio}{*}\colon A\to [0, \infty)$
is 
an absolute value on 
$A$;

\item 
For 
$\yoratio\in (1, \infty)$,
and for every absolute value 
$\lVert *\rVert$
on 
$A$,
the function 
$\yonval{\lVert *\rVert}{\yoratio}\colon A\to \rr\sqcup\{\infty\}$
is a valuation on 
$A$.
\end{enumerate}
%enuend
\end{prop}
%propend
\begin{proof}
Both assertions follow directly from the definitions and the identities 
$\yoratio^{-(a+b)}=\yoratio^{-a}\yoratio^{-b}$
and 
$\yoratio^{-\min\{a,b\}}=\max\{\yoratio^{-a},\yoratio^{-b}\}$.
\end{proof}

We use both additive valuations and absolute values below. 
Additive valuations are denoted by symbols such as
$v$
and
$w$,
whereas absolute values are denoted by symbols such as 
$\lvert *\rvert$
and
$\lVert *\rVert$.

Let
$(A,v)$
be a valued domain and let
$K$
be its field of fractions, 
equipped with the extension of
$v$.
The set 
$\yovring{A}{v}=\{\, x\in K\mid 0\le v(x)\, \}$
is a valuation ring, and 
$\yovideal{A}{v}=\{\, x\in K\mid 0< v(x)\, \}$
is its maximal ideal. 
We denote by 
$\yorcf{A}{v}$
the residue field 
$\yovring{A}{v}/\yovideal{A}{v}$,
also called the 
\emph{residue class field of
$(A,v)$}.
We also denote by  
\[
\yorcff_{A}\colon \yovring{A}{v}\to \yorcf{A}{v}
\]
the canonical projection. 
We simply represent it as 
$\yorcff$
when no confusion
can arise. 
We say that a subset 
$\yorep$
of 
$K$
is 
a 
\emph{complete system of representatives of
$\yorcf{A}{v}$}
if 
$\yorep\yosub \yovring{A}{v}$,
$0\in \yorep$,
and 
$\yorcff|_{\yorep}\colon \yorep\to  \yorcf{A}{v}$
is 
bijective. 
Note that our definition requires 
$0\in \yorep$.

\subsection{Constructions of valued fields}\label{subsec:const}

We recall Hahn fields and their
$p$-adic
analogues. 
For background, see \cite{MR3782290,MR4367483}. 
Most of the results recalled below are taken from 
\cite{MR1225257,MR0554237}.

\subsubsection{Hahn rings and fields}

A non-empty subset 
$S$
of 
$\rr$
is 
said to be 
\emph{well-ordered} 
if 
every non-empty subset of 
$S$
has a minimum.

We denote by 
$\yogset$
the 
set of all subgroups
$G$
 of 
$\rr$
with 
$1\in G$
(equivalently, 
$\zz\yosub G\yosub \rr$).
For the sake of convenience, 
we only consider the setting
where
 $G\in \yogset$
in this paper.

Now we review 
the construction 
of the Hahn fields  in 
\cite{MR1225257}. 
Let 
$G\in \yogset$
and 
$A$
be a commutative ring. 
For a map 
$a\colon G\to A$,
we define the support  
$\supp(a)$
of 
$a$
by 
the set 
$\{\, x\in G\mid a(x)\neq 0\, \}$.
We denote by 
$\yoha{G}{A}$
the set of all 
$a\colon G\to A$
such that 
$\supp(a)$
is well-ordered. 
We often symbolically  represent 
$a\in \yoha{G}{A}$
as 
\[
a=\sum_{g\in G}a(g)t^{g}, 
\]
where
$t$
is an indeterminate. 
For every pair 
$a, b \in \yoha{G}{A}$,
we define 
$a+b$
by 
\[
(a+b)(x)=a(x)+b(x). 
\]
We also define 
the multiplication 
$ab\colon G\to A$
by 
\[
ab=\sum_{g\in G}\left(\sum_{i, j\in G, i+j=g}a_{i}b_{j}\right)t^{g}. 
\]
For the zero map, 
we put 
\[
\min\supp(0)=\infty.
\]
Define a map 
$\yohav{G}{A}$
on 
$\yoha{G}{A}$
by 
$\yohav{G}{A}(a)=\min \supp(a)$.
Since 
$\supp(a)$
is well-ordered for every non-zero 
$a$,
the minimum 
$\min\supp(a)$
actually exists. 
Note that 
$A$
becomes 
 a subring 
of 
$\yoha{G}{A}$.

%statebody
%propbegin
\begin{prop}\label{prop:prophahn}
Let 
$G\in \yogset$,
and 
$\yosmf{k}$
be a field. 
Then 
$(\yoha{G}{\yosmf{k}}, \yohav{G}{\yosmf{k}})$
becomes 
a valued field and  it 
satisfies 
$\yorcf{\yoha{G}{\yosmf{k}}}{\yohav{G}{\yosmf{k}}}=\yosmf{k}$.
\end{prop}
%propend
%proof
%proofbegin
\begin{proof}
See 
\cite[Corollary 1]{MR1225257}. 
\end{proof}
%proofend

We call 
$(\yoha{G}{A}, \yohav{G}{A})$
the 
\emph{Hahn ring  associated with 
$G$
and
$A$}
and call
it 
 the 
 \emph{Hahn field} 
 if 
$A$
is a field. 
Note that 
in general, 
we can define the Hahn fields even if 
$G$
is a linearly ordered Abelian group
(see \cite{MR1225257}).

\subsubsection{The
$p$-adic
Hahn fields}
A 
$p$-adic
analogue of the Hahn fields was
 first 
introduced in
 \cite{MR1225257}. 
Let us review 
a
construction. 
A field 
$\yosmf{k}$
of 
 characteristic 
$p$
is said to be 
\emph{perfect} 
if 
$p=0$,
or 
$p>0$
and 
every element
of
$\yosmf{k}$
has a 
$p$-th
root
in 
$\yosmf{k}$.
The following proposition states the 
existence of rings of Witt vectors.
The proof is presented in 
\cite{MR0554237}.

%state body 
%propbegin
\begin{prop}\label{prop:witt}
Let 
$\yosmf{k}$
be a perfect field of characteristic 
$p>0$.
Then, up to a unique valuation-preserving isomorphism inducing the identity on 
$\yosmf{k}$,
there exists a unique valued ring 
$(A, v)$
of characteristic 
$0$
such that
$A$
is complete, 
$v(A\setminus\{0\})=\zz_{\ge 0}$,
$v(p)=1$,
and 
$\yorcf{A}{v}=\yosmf{k}$.
\end{prop}
%propend

Let
$\yosmf{k}$
be a perfect field of characteristic
$p>0$.
We write 
$\yowitto{\yosmf{k}}$
 for the ring in 
Proposition
 \ref{prop:witt} 
and 
$\yowitt{\yosmf{k}}$
 for its field of fractions. 
The valuation 
$\yowiv{\yosmf{k}}$
on 
$\yowitto{\yosmf{k}}$
extends canonically to 
$\yowitt{\yosmf{k}}$;
we use the same symbol for this extension. 
The ring 
$(\yowitto{\yosmf{k}}, \yowiv{\yosmf{k}})$
is called the 
\emph{ring of Witt vectors associated with 
$\yosmf{k}$}.
Notice that 
for every prime 
$p$,
we have 
$\yowitto{\yogf{p}}=\zz_{p}$
and 
$\yowitt{\yogf{p}}=\yopnbf{p}$,
and the valuation 
$\yowiv{\yogf{p}}$
coincides with 
the
 $p$-adic
valuation
$\yopnbv{p}$.

Now we discuss
the 
$p$-adic
analogue of a Hahn field, 
defined as a quotient of a Hahn ring. 
For 
$G\in \yogset$,
and 
for a 
perfect 
field 
$\yosmf{k}$
 of  characteristic 
 $p>0$,
we define 
a subset 
$\yopideal{G}{\yosmf{k}}$
of 
$\yoha{G}{\yowitto{\yosmf{k}}}$
by 
the 
set of all 
$\alpha=\sum_{g\in G}\alpha_{g}t^{g}\in 
\yoha{G}{\yowitto{\yosmf{k}}}$
such that 
$\sum_{n\in \zz}\alpha_{g+n}p^{n}=0$
in the field of fractions 
$\yowitt{\yosmf{k}}$
for every 
$g\in G$.
For fixed
$g$,
this sum is well-defined: 
the set of integers
$n$
for which
$\alpha_{g+n}\neq0$
is bounded below, 
and the resulting series converges
$p$-adically
in
$\yowitt{\yosmf{k}}$.

%statebody 
%propbegin
\begin{prop}\label{prop:p-adic}
For every 
$G\in \yogset$,
and 
for every  
perfect field 
$\yosmf{k}$
of characteristic 
$p>0$,
the set 
$\yopideal{G}{\yosmf{k}}$
is 
an ideal of the ring 
$\yoha{G}{\yowitto{\yosmf{k}}}$,
and 
$\yoha{G}{\yowitto{\yosmf{k}}}/\yopideal{G}{\yosmf{k}}$
is a 
field
(i.e.,
$\yopideal{G}{\yosmf{k}}$
is maximal).
\end{prop}
%propend
%proof 
\begin{proof}
See 
\cite[Proposition 3]{MR1225257}
and 
\cite[Corollary 3]{MR1225257}. 
\end{proof}
%proofend

For subsets 
$A$,
$B$
of 
$\rr$,
we define 
$A+B=\{\, a+b\mid a\in A, b\in B\, \}$.

%state body 
%lembegin
\begin{lem}\label{lem:repbeta}
Let 
$G\in \yogset$,
$p$
be a prime, 
and 
$\yosmf{k}$
be a perfect field 
of characteristic 
$p>0$.
Let 
$\yorep\yosub \yowitto{\yosmf{k}}$
be a 
complete system of representatives of 
$\yosmf{k}$
in
$\yowitto{\yosmf{k}}$.
Then
 every element 
$\alpha=\sum_{g\in G}\alpha_{g}t^{g}\in 
\yoha{G}{\yowitto{\yosmf{k}}}$
is equivalent to 
an element 
$\beta=\sum_{g\in G}\beta_{g}t^{g}$
modulo 
$\yopideal{G}{\yosmf{k}}$,
where 
$\beta_{g}$
 is in 
$\yorep$.
In addition, the family
$\{\beta_{g}\}_{g\in G}$
is unique and 
$\supp(\beta)\yosub \supp(\alpha)+\zz_{\ge 0}$.
\end{lem}
%lemend
%proof
\begin{proof}
See 
\cite[Proposition 4]{MR1225257}. 
\end{proof}
%proofend

%dfbe
\begin{df}\label{df:bbbb}
Let 
$G\in \yogset$,
$p$
be a prime, 
and let 
$\yosmf{k}$
be a perfect field of characteristic 
$p>0$.
Fix a complete system 
$\yorep\yosub\yowitto{\yosmf{k}}$
of representatives of 
$\yosmf{k}$.
For every 
$a\in \yoha{G}{\yowitto{\yosmf{k}}}$,
Lemma \ref{lem:repbeta} gives a unique series 
with coefficients in 
$\yorep$
that is equivalent to
 $a$
 modulo 
$\yopideal{G}{\yosmf{k}}$.
We denote this series by 
$\yostan{G}{\yosmf{k}}{\yorep}(a)$
and call it the 
\emph{standard representation of
 $a$
with respect to
$\yorep$}.
It satisfies 
$\supp(\yostan{G}{\yosmf{k}}{\yorep}(a))\yosub \supp(a)+\zz_{\ge 0}$.
\end{df}
%dfed

Let 
$\yopf{G}{\yosmf{k}}{p}$
denote the field
\[
\yoha{G}{\yowitto{\yosmf{k}}}/\yopideal{G}{\yosmf{k}},
\]
and let
\[
\yoproj\colon \yoha{G}{\yowitto{\yosmf{k}}}\to 
\yopf{G}{\yosmf{k}}{p}
\]
be the canonical projection. 

%state body 
%propbegin
\begin{prop}\label{prop:indS}
Let 
$G\in \yogset$,
and 
$\yosmf{k}$
be a perfect field
of characteristic 
$p>0$.
Take a complete system 
$\yorep\yosub \yowitto{\yosmf{k}}$
of representatives of 
$\yosmf{k}$
in 
$\yowitto{\yosmf{k}}$.
Then the following 
statements are true:
%%enube
\begin{enumerate}[label=\textup{(\arabic*)}]

\item 
For every 
$z\in \yoha{G}{\yowitto{\yosmf{k}}}$,
 the value 
$\min \supp(\yostan{G}{\yosmf{k}}{\yorep}(z))$
is 
 independent of 
the choice of 
$\yorep$.

\item 

We define 
the map 
\[
\yopfv{G}{\yosmf{k}}{p}\colon
\yopf{G}{\yosmf{k}}{p}\to G\sqcup\{\infty\}
\]
by 
\[
\yopfv{G}{\yosmf{k}}{p}(x)=\min \supp(\yostan{G}{\yosmf{k}}{\yorep}(z)), 
\]
where 
$z$
belongs to 
$\yoha{G}{\yowitto{\yosmf{k}}}$
and satisfies 
$\yoproj(z)=x$.
Then 
$\yopfv{G}{\yosmf{k}}{p}$
is 
a valuation on 
$\yopf{G}{\yosmf{k}}{p}$.
\end{enumerate}
%enuend
\end{prop}
%\propend 
%proofb
\begin{proof}
See \cite[Proposition 5]{MR1225257}. 
\end{proof}
%proofe

We call 
the valued field 
\[
(\yopf{G}{\yosmf{k}}{p}, \yopfv{G}{\yosmf{k}}{p})
\]
the 
\emph{$p$-adic
Mal'cev--Neumann field}
or 
\emph{$p$-adic
Hahn field}. 
In particular, 
$(\yopf{\zz}{\yogf{p}}{p}, \yopfv{\zz}{\yogf{p}}{p})$
is nothing but 
the 
field
$(\yopnbf{p}, \yopnbv{p})$
 of 
$p$-adic
numbers. 

To consider 
characteristics 
of a valued field and 
its residue class field, 
we additionally 
define 
$\yoqpset$
by the 
 set of all pairs
$(q, p)$
such that
$q$
and 
$p$
are 
$0$
or a prime 
satisfying either of the following conditions: 
%enubegin
\begin{enumerate}[label=\textup{(Q\arabic*)}]

\item\label{item:condq1} 
$q=p$,

\item\label{item:condq2} 
$q=0$
and 
$0<p$.
\end{enumerate}
Note that 
$(q, p)\in \yoqpset$
 satisfies 
\ref{item:condq2}  if and only if 
$q\neq p$.
%enuend

In order to discuss 
$p$-adic
and 
ordinary 
Hahn fields in 
a
unified manner, 
we introduce the following 
notation.

%dfbegin
\begin{df}\label{df:uk}
Let 
$G\in \yogset$,
$(q, p)\in \yoqpset$,
and 
let 
$\yosmf{k}$
 be  a 
 perfect field 
 of characteristic 
 $p$.
We define a
field 
$\youhaq{G}{\yosmf{k}}{q}{p}$
 by
\[
\youhaq{G}{\yosmf{k}}{q}{p}
=\begin{cases}
\yoha{G}{\yosmf{k}} & \text{if
$q=p$;}\\
\yopf{G}{\yosmf{k}}{p} & \text{if
$q\neq p$.}
\end{cases}
\]
We also define a valuation
$\youhavq{G}{\yosmf{k}}{q}{p}$
on 
$\youhaq{G}{\yosmf{k}}{q}{p}$
 by 
\[
\youhavq{G}{\yosmf{k}}{q}{p}
=\begin{cases}
\yohav{G}{\yosmf{k}} & \text{if
$q=p$;}\\
\yopfv{G}{\yosmf{k}}{p} & \text{if
$q\neq p$.}
\end{cases}
\]
\end{df}
%dfed

A metric space is called \emph{spherically complete} if every decreasing
sequence of non-empty closed or open balls has non-empty intersection.

\begin{prop}\label{prop:hahnspherical}
Let 
$G\in\yogset$,
$(q,p)\in\yoqpset$,
and let 
$\yosmf{k}$
 be a perfect
field of characteristic 
$p$.
Then
\begin{enumerate}[label=\textup{(\arabic*)}]
\item
$\youhavq{G}{\yosmf{k}}{q}{p}
(\youhaq{G}{\yosmf{k}}{q}{p}\setminus\{0\})=G$;
\item
$\yorcf{\youhaq{G}{\yosmf{k}}{q}{p}}
{\youhavq{G}{\yosmf{k}}{q}{p}}=\yosmf{k}$;
\item
$(\youhaq{G}{\yosmf{k}}{q}{p},
\youhavq{G}{\yosmf{k}}{q}{p})$
is spherically complete.
\end{enumerate}
\end{prop}
\begin{proof}
The first two assertions follow from the construction.
The last follows
from \cite[Theorem 1]{MR1225257} and \cite[Theorem 4]{MR0006161}; see also
\cite[Theorem 6.11]{MR3782290}.
\end{proof}

%statebody
%propbegin
\begin{prop}\label{prop:embedcoef}
Let 
$G\in \yogset$,
$(q, p)\in \yoqpset$,
and let 
$\yosmf{k}$
be a perfect field of characteristic 
$p$.
Assume that 
$q\neq p$.
Then 
$(\youhaq{G}{\yosmf{k}}{q}{p},
\youhavq{G}{\yosmf{k}}{q}{p})$
is a valued field extension of 
$(\yowitt{\yosmf{k}}, \yowiv{\yosmf{k}})$.
In particular, 
we can regard 
$\yowitto{\yosmf{k}}$
as a subring of 
\[
\yovring{\youhaq{G}{\yosmf{k}}{q}{p}}
{\youhavq{G}{\yosmf{k}}{q}{p}}.
\]
\end{prop}
%propend
%proofbegin
\begin{proof}
Since 
$q\neq p$,
we have 
$q=0$
and 
$p>0$.
By the definition of 
$\yogset$,
we have 
$\zz\yosub G$.
Since 
$\yowitt{\yosmf{k}}=\youhaq{\zz}{\yosmf{k}}{q}{p}$,
the inclusion 
$\zz\yosub G$
induces a valued field embedding
\[
(\yowitt{\yosmf{k}}, \yowiv{\yosmf{k}})
\hookrightarrow
(\youhaq{G}{\yosmf{k}}{q}{p},
\youhavq{G}{\yosmf{k}}{q}{p}).
\]
Thus 
$(\youhaq{G}{\yosmf{k}}{q}{p},
\youhavq{G}{\yosmf{k}}{q}{p})$
is a valued field extension of 
$(\yowitt{\yosmf{k}}, \yowiv{\yosmf{k}})$.
The last assertion follows from the inclusion 
$\yowitto{\yosmf{k}}\yosub \yowitt{\yosmf{k}}$.
\end{proof}
%proofend

%dfbe
\begin{df}\label{df:tau}
Let
$G\in \yogset$,
$(q, p)\in \yoqpset$,
and let
$\yosmf{k}$
be a perfect field of characteristic
$p$.
We fix the following conventions for the rest of the paper. 
If
$q=p$,
we use the canonical system of representatives 
$\yorep=\yosmf{k}\subset\yoha{G}{\yosmf{k}}$
and put 
$\youit^{g}=t^{g}$
for every
$g\in G$.
If
$q\neq p$,
we fix a complete system 
$\yorep\subset\yowitto{\yosmf{k}}$
of representatives of
$\yosmf{k}$,
regard it as a subset of
$\yopf{G}{\yosmf{k}}{p}$
through 
Proposition \ref{prop:embedcoef}, and put
$\youit^{g}=\yoproj(t^{g})$
 for every 
$g\in G$.
Since
$t-p\in\yopideal{G}{\yosmf{k}}$,
we have 
$\youit=\yoproj(t)=p$
 in this case. 
For
$y\in\youhaq{G}{\yosmf{k}}{q}{p}$,
there is a unique expansion 
\[
y=\sum_{g\in G}\yocoef{y}{g}\youit^{g}
\]
with
$\yocoef{y}{g}\in\yorep$.
When
$q=p$,
this is the usual Hahn-series expansion. 
When
$q\neq p$,
the displayed equality means
\[
y=\yoproj\left(\sum_{g\in G}\yocoef{y}{g}t^{g}\right),
\]
where the series on the right is the unique standard representation
supplied by Lemma \ref{lem:repbeta}. 
The coefficients in the latter case depend on the fixed system
$\yorep$,
which will not be changed below.
\end{df}
%dfend

\subsubsection{Levi--Civita fields}
We next discuss Levi--Civita fields and 
$p$-adic
 Levi--Civita fields, 
which will be 
 used in 
Section 
\ref{sec:Urysohn}.

Let 
$G\in \yogset$,
and 
$\yosmf{k}$
 be a field. 
For 
 $f\in \yoha{G}{\yosmf{k}}$,
we consider the following condition. 
\begin{enumerate}[label=\textup{(Fin)}]
\item\label{item:19finc}
For every 
$n\in \zz$,
the set
$\supp(f)\cap (-\infty, n]$
is finite. 
\end{enumerate}
In what follows, 
for the zero element, 
we put 
\[
\min \supp(0)=\infty.
\]
We use the convention that 
$\infty>g$
for every
$g\in G$.
We denote by
  $\yolc{G}{\yosmf{k}}$
the set of all 
$f\in \yoha{G}{\yosmf{k}}$
satisfying 
the condition 
\ref{item:19finc}. 
For the next lemma, 
we refer the reader to 
\cite[Theorem 3.18]{MR3782290}. 
\begin{lem}\label{lem:levicivita}
Let
$G\in \yogset$,
$\yosmf{k}$
be  a field. 
Then 
the set 
$\yolc{G}{\yosmf{k}}$
is a subfield of 
$\yoha{G}{\yosmf{k}}$.
\end{lem}
We call 
$\yolc{G}{\yosmf{k}}$
the 
\emph{Levi--Civita field associated with
$G$
and
$\yosmf{k}$}.

Fix
$(q, p)\in \yoqpset$
with
$q\neq p$
and 
assume that
$\yosmf{k}$
has 
characteristic 
$p>0$.
Before defining a 
$p$-adic
analogue
of Levi--Civita fields, 
we define a subset 
$\yopd{G}{\yosmf{k}}$
of 
$\yoha{G}{\yowitto{\yosmf{k}}}$
by the set of all 
$f\in \yoha{G}{\yowitto{\yosmf{k}}}$
satisfying the
condition 
\ref{item:19finc}. 
We define
$\yoplc{G}{\yosmf{k}}{p}$
by 
$\yoplc{G}{\yosmf{k}}{p}=
\yoproj(\yopd{G}{\yosmf{k}})$, 
where 
\[
\yoproj\colon \yoha{G}{\yowitto{\yosmf{k}}}\to 
\yopf{G}{\yosmf{k}}{p}=\yoha{G}{\yowitto{\yosmf{k}}}/\yopideal{G}{\yosmf{k}}
\]
is 
the canonical projection. 

%lembe 
\begin{lem}\label{lem:dfield}
Let 
$G\in \yogset$,
$p$
be a prime, 
and 
$\yosmf{k}$
be a perfect field of 
characteristic 
$p>0$.
Then 
the following 
statements are true:
%enube
\begin{enumerate}[label=\textup{(\arabic*)}]
\item\label{item:ddd1}
For every complete system 
$\yorep\yosub \yowitto{\yosmf{k}}$
of representatives of 
$\yosmf{k}$,
and for 
every
$a\in \yopd{G}{\yosmf{k}}$,
the member 
 \[
 f=\yostan{G}{\yosmf{k}}{\yorep}(a)
 \in \yoha{G}{\yowitto{\yosmf{k}}}
 \]
 satisfies the condition 
\ref{item:19finc}
and 
$f(g)\in \yorep$
for all
 $g\in G$.

\item\label{item:ddd2}
The set
$\yopd{G}{\yosmf{k}}$
is a 
subring of 
$\yoha{G}{\yowitto{\yosmf{k}}}$;

\item\label{item:ddd3}
If 
$a\in \yopd{G}{\yosmf{k}}$
satisfies 
$\min \supp(a)>0$,
then 
\[
(1-a)^{-1}\in \yopd{G}{\yosmf{k}}.
\]

\item\label{item:ddd4}
For every
$a\in \yopd{G}{\yosmf{k}}$
with 
$\yoproj(a)\neq 0$,
there exists
$b\in \yopd{G}{\yosmf{k}}$
such that 
$ab$
is equivalent to 
$1$
modulo 
$\yopideal{G}{\yosmf{k}}$.

\end{enumerate}
%enued
\end{lem}
%lemed
%proofbe 
\begin{proof}
Fix a complete system 
$\yorep\yosub \yowitto{\yosmf{k}}$
of representatives of 
$\yosmf{k}$.
First we prove 
\ref{item:ddd1}. 
Put 
$f=\yostan{G}{\yosmf{k}}{\yorep}(a)$.
By Lemma 
\ref{lem:repbeta}, 
we see that 
$f(g)\in \yorep$
for all 
$g\in G$.
Put 
$A=\supp(a)$
and 
$F=\supp(f)$.
Lemma 
\ref{lem:repbeta} 
also 
shows that 
$F\yosub A+\zz_{\ge 0}$.
Due to this relation, 
since
$A$
satisfies 
the condition 
\ref{item:19finc}, 
so does
$F$.
Hence 
the statement 
\ref{item:ddd1} 
is true. 

Next we prove 
\ref{item:ddd2}. 
By the definitions of 
$\yopd{G}{\yosmf{k}}$
and 
$\yolc{G}{\yowitt{\yosmf{k}}}$,
we have 
$\yopd{G}{\yosmf{k}}=\yoha{G}{\yowitto{\yosmf{k}}}\cap 
\yolc{G}{\yowitt{\yosmf{k}}}$
(pay attention to the difference between
$\yowitto{\yosmf{k}}$
and
$\yowitt{\yosmf{k}}$
appearing in
$\yoha{G}{\yowitto{\yosmf{k}}}$
and 
 $\yolc{G}{\yowitt{\yosmf{k}}}$,
 respectively). 
 We directly check that 
 $\yopd{G}{\yosmf{k}}$
 is closed under the ring operations. 
Let 
$a,b\in \yopd{G}{\yosmf{k}}$.
Then 
$\supp(a+b)\yosub \supp(a)\cup \supp(b)$,
and hence 
$a+b$
 satisfies 
\ref{item:19finc}. 
The same is clear for 
$-a$.
For multiplication, 
we have 
$\supp(ab)\yosub \supp(a)+\supp(b)$.
Hence
$ab$
satisfies 
\ref{item:19finc}. 
Therefore 
$\yopd{G}{\yosmf{k}}$
is a subring of 
the ring
$\yoha{G}{\yowitto{\yosmf{k}}}$.

Now we show \ref{item:ddd3}. 
Put 
$m=\min \supp(a)$.
If
$a=0$,
then the assertion is trivial. 
Assume that
$a\neq 0$.
As in \cite{MR1225257}, 
we have 
\[
(1-a)^{-1}=1+a+a^{2}+a^{3}+\cdots
\]
in 
$\yoha{G}{\yowitto{\yosmf{k}}}$.
For every 
$i\in \zz_{\ge 0}$,
put 
$A_{i}=\supp(a^{i})$.
In this case, we have 
$i\cdot m\le \min A_{i}$
for all 
$i\in \zz_{\ge 1}$.
Put 
$B=\bigcup_{i\in \zz_{\ge 1}} A_{i}$.
For a fixed
$n\in\zz$,
only finitely many indices
$i\ge1$
satisfy
$im\le n$.
Moreover, each set
$A_{i}\cap(-\infty,n]$
is finite by 
\ref{item:ddd2}. 
It follows that
$B\cap(-\infty,n]$
is finite. 
Put 
$E=\supp((1-a)^{-1})$.
Then 
$E\yosub \{0\}\cup B$,
and hence 
$E$
satisfies 
that 
$(-\infty, n]\cap E$
is finite 
for all 
$n\in \zz$.
This implies that 
$(1-a)^{-1}\in \yopd{G}{\yosmf{k}}$.

We shall prove \ref{item:ddd4}. 
Put 
$f=\yostan{G}{\yosmf{k}}{\yorep}(a)$.
Since 
$\yoproj(a)\neq 0$,
we have 
$f\neq 0$.
We only need to show 
that 
$f$
is invertible 
in 
$\yopd{G}{\yosmf{k}}$.
We represent
 $f$
 as
$f=\sum_{g\in G}z_{g}t^{g}$.
Take
 $m=\min \supp(f)$
and 
put 
$y=\sum_{m<g}z_{g}t^{g-m}$.
Since 
$\yowiv{\yosmf{k}}(z_{m})=0$,
we see that 
$z_{m}$
is invertible in 
$\yowitto{\yosmf{k}}$.
Then 
\[
f=z_{m}t^{m}(1-(-z_{m}^{-1}y)). 
\]
If 
$y\neq 0$,
then 
\[
\min \supp(z_{m}^{-1}y)>0.
\]
Since 
$z_{m}$
and 
$t^{m}$
 are invertible in 
$\yopd{G}{\yosmf{k}}$,
and since 
$1-(-z_{m}^{-1}y)$
is invertible by 
\ref{item:ddd3} when 
$y\neq 0$
and trivially invertible when
 $y=0$,
we see that 
$f$
is invertible in 
$\yopd{G}{\yosmf{k}}$.
Let 
$b$
 be the inverse of 
 $f$
 in 
$\yopd{G}{\yosmf{k}}$.
Since
 $a$
 is equivalent to 
 $f$
 modulo 
$\yopideal{G}{\yosmf{k}}$,
the product
 $ab$
 is equivalent to 
 $1$
 modulo 
$\yopideal{G}{\yosmf{k}}$.
\end{proof}
%proofed

Lemma \ref{lem:dfield} yields the following corollary.

%lembe
\begin{cor}\label{cor:216}
Let 
$G\in \yogset$,
$p$
be a prime, 
and 
$\yosmf{k}$
be a perfect field of 
characteristic 
$p>0$.
Then the following statements are true: 
\begin{enumerate}[label=\textup{(\arabic*)}]

\item\label{item:ffff}
For every complete system 
$\yorep\yosub \yowitto{\yosmf{k}}$
of representatives of 
$\yosmf{k}$
in
$\yowitto{\yosmf{k}}$,
which we regard as a complete system of representatives
of 
$\yosmf{k}$
in
$\yopf{G}{\yosmf{k}}{p}$
via the canonical embedding, 
the set 
$\yoplc{G}{\yosmf{k}}{p}$
is equal to 
the set 
\[
\yoproj(\yostan{G}{\yosmf{k}}{\yorep}(\yopd{G}{\yosmf{k}})). 
\]
Moreover, 
for every
$a\in \yoplc{G}{\yosmf{k}}{p}$,
there exists a unique 
$f\in \yostan{G}{\yosmf{k}}{\yorep}(\yopd{G}{\yosmf{k}})$
such that
$a=\yoproj(f)$.
\item\label{item:gggg}
The set 
$\yoplc{G}{\yosmf{k}}{p}$
is a subfield of 
$\yopf{G}{\yosmf{k}}{p}$.
\end{enumerate}

\end{cor}
%lemed
%proofbegin
\begin{proof}
We first prove 
\ref{item:ffff}. 
Let 
$a\in \yoplc{G}{\yosmf{k}}{p}$.
Then there exists 
$b\in \yopd{G}{\yosmf{k}}$
such that 
$a=\yoproj(b)$.
By Lemma \ref{lem:dfield}, 
the element 
$\yostan{G}{\yosmf{k}}{\yorep}(b)$
belongs to 
$\yostan{G}{\yosmf{k}}{\yorep}(\yopd{G}{\yosmf{k}})$,
and 
\[
a=\yoproj(\yostan{G}{\yosmf{k}}{\yorep}(b)).
\]
This proves the stated equality. 
The uniqueness follows from the uniqueness in 
Lemma \ref{lem:repbeta}. 

We next prove 
\ref{item:gggg}. 
By Lemma \ref{lem:dfield}, 
the set 
$\yopd{G}{\yosmf{k}}$
is a subring of 
$\yoha{G}{\yowitto{\yosmf{k}}}$.
Therefore 
$\yoplc{G}{\yosmf{k}}{p}=\yoproj(\yopd{G}{\yosmf{k}})$
is a subring of 
$\yopf{G}{\yosmf{k}}{p}$.
Let 
$a\in \yoplc{G}{\yosmf{k}}{p}$
be non-zero. 
Take 
$b\in \yopd{G}{\yosmf{k}}$
such that 
$a=\yoproj(b)$.
Then 
$\yoproj(b)\neq 0$.
By Lemma \ref{lem:dfield}, 
there exists 
$c\in \yopd{G}{\yosmf{k}}$
such that 
$bc$
is equivalent to 
$1$
modulo 
$\yopideal{G}{\yosmf{k}}$.
Thus 
$a^{-1}=\yoproj(c)$
belongs to 
$\yoplc{G}{\yosmf{k}}{p}$.
This proves that 
$\yoplc{G}{\yosmf{k}}{p}$
is a subfield of 
$\yopf{G}{\yosmf{k}}{p}$.
\end{proof}
%proofend

We call 
$\yoplc{G}{\yosmf{k}}{p}$
the 
\emph{$p$-adic
Levi--Civita field
associated with 
$G$
and 
$\yosmf{k}$}.

To use
$p$-adic
and 
ordinary 
Levi--Civita
fields in a 
unified way, 
we 
make the following 
definition. 

\begin{df}\label{df:uklc}
Let 
$G\in \yogset$,
$(q, p)\in \yoqpset$,
and 
$\yosmf{k}$
be a perfect field of characteristic 
$p$.
We define 
$\youulc{G}{\yosmf{k}}{q}{p}$
by
\[
\youulc{G}{\yosmf{k}}{q}{p}
=\begin{cases}
\yolc{G}{\yosmf{k}} & \text{if
$q=p$;}\\
\yoplc{G}{\yosmf{k}}{p} & \text{if
$q\neq p$.}
\end{cases}
\]
\end{df}

%state body 
\begin{prop}\label{prop:pppsubsub}
Let 
$G\in \yogset$,
$(q, p)\in \yoqpset$,
and 
$\yosmf{k}$
be a perfect field of characteristic 
$p$.
Then the set 
$\youulc{G}{\yosmf{k}}{q}{p}$
is a 
subfield of 
$\youhaq{G}{\yosmf{k}}{q}{p}$.
\end{prop}
%propend
%proofbegin
\begin{proof}
The case
$q=p$
follows from 
Lemma 
\ref{lem:levicivita}. 
The case
$q\neq p$
follows from 
Corollary 
\ref{cor:216}. 
\end{proof}
%proofend

We simply represent the restriction
$\youhavq{G}{\yosmf{k}}{q}{p}|_{\youulc{G}{\yosmf{k}}{q}{p}}$
by the same symbol 
$\youhavq{G}{\yosmf{k}}{q}{p}$.
In this setting, 
the field 
$(\youulc{G}{\yosmf{k}}{q}{p}, \youhavq{G}{\yosmf{k}}{q}{p})$
becomes a valued subfield of 
$(\youhaq{G}{\yosmf{k}}{q}{p}, \youhavq{G}{\yosmf{k}}{q}{p})$.

\begin{prop}\label{prop:primevaluedextension}
Let
$G\in\yogset$.
The following statements hold:
\begin{enumerate}[label=\textup{(\arabic*)}]
\item
If
$\yosmf{k}$
is a field of characteristic
$0$,
then
$(\youulc{G}{\yosmf{k}}{0}{0},
\youhavq{G}{\yosmf{k}}{0}{0})$
is a valued field extension of
$\qq$
equipped with the trivial valuation.
\item
If
$p$
is a prime and
$\yosmf{k}$
is a perfect field of characteristic
$p$,
then
$(\youulc{G}{\yosmf{k}}{0}{p},
\youhavq{G}{\yosmf{k}}{0}{p})$
is a valued field extension of
$(\yopnbf{p},\yopnbv{p})$.
\end{enumerate}
\end{prop}
\begin{proof}
The first assertion follows by regarding the prime subfield
$\qq\subset\yosmf{k}$
as the constant-coefficient subfield of
$\yolc{G}{\yosmf{k}}$.

For the second assertion, by the functoriality of Witt vectors
\cite[Proposition 10, p.~39]{MR0554237},
the inclusion
$\yogf{p}\subset\yosmf{k}$
induces a valued field embedding
\[
(\yopnbf{p},\yopnbv{p})
=(\yowitt{\yogf{p}},\yowiv{\yogf{p}})
\hookrightarrow
(\yowitt{\yosmf{k}},\yowiv{\yosmf{k}}).
\]
As in Proposition \ref{prop:embedcoef}, we identify
$\yowitt{\yosmf{k}}$
with
$\youhaq{\zz}{\yosmf{k}}{0}{p}$.
Every well-ordered subset of
$\zz$
satisfies
\ref{item:19finc}, and hence
\[
\youhaq{\zz}{\yosmf{k}}{0}{p}
=\youulc{\zz}{\yosmf{k}}{0}{p}.
\]
The inclusion
$\zz\subset G$
therefore induces a valued field embedding
\[
\yowitt{\yosmf{k}}
\hookrightarrow
\youulc{G}{\yosmf{k}}{0}{p}.
\]
The composite of these embeddings proves the second assertion.
\end{proof}

%%%%%%%%%%%%%%%%%%%%
%%%%%%%%%%%%%%%%%%%%%%%%%%%%%

\section{Algebraic structures on 
non-Archimedean Urysohn universal 
metric spaces}\label{sec:Urysohn}
In this section, 
we show 
that 
the Urysohn universal ultrametric spaces associated with range sets
of the form
$\{0\}\cup\{\,\yoratio^{-g}\mid g\in G\,\}$,
where
$\yoratio>1$
and
$G\in\yogset$,
admit
a valued field 
extending 
a given prime valued field.
Such an algebraic structure
is realized by a
$p$-adic
or 
ordinary 
Levi--Civita 
field. 
Some parts of this section overlap with 
the material on petaloid spaces in 
\cite{Ishiki2023Ury}. 
We recall them here in order to make the present paper 
self-contained 
and to fix the notation used below. 

\subsection{Urysohn universal ultrametric spaces and petaloid structure}

For a class
$\yochara{C}$
of metric spaces, 
a metric space
$(X, d)$
is said to be 
\emph{$\yochara{C}$-injective} if 
for every pair
$(A, d_{A})$
and
$(B, d_{B})$
in
$\yochara{C}$
and 
for every pair of isometric embeddings 
$\phi\colon (A, d_{A})\to (B, d_{B})$
and 
$\psi\colon (A, d_{A})\to (X, d)$,
there exists an isometric embedding 
$\theta\colon (B, d_{B})\to (X, d)$
such that 
$\theta\circ \phi=\psi$.
We denote by 
$\yofin$
(resp.~$\youfin(R)$
for a range set
$R$)
the class of 
all finite metric spaces
(resp.~all finite 
$R$-valued
ultrametric spaces). 
There exists a 
separable complete 
$\yofin$-injective
metric space
$(\yourysp, \yourydis)$,
unique up to isometry, 
and this space is called 
the 
\emph{Urysohn universal metric space}
(see \cite{Ury1927} and \cite{MR2435148}).

Similarly, 
if 
$R$
is countable, 
there exists 
a separable complete 
$R$-valued
$\youfin(R)$-injective
ultrametric space, 
unique up to isometry, 
and it is called the 
\emph{$R$-Urysohn
universal ultrametric space}, 
which is a non-Archimedean 
analogue of
$(\yourysp, \yourydis)$
(see
\cite{MR2754373} and \cite{MR2667917}). 
Note that if
$R$
is uncountable, 
then every injective
$R$-valued
ultrametric space 
is non-separable
(see \cite{MR2754373}). 
Thus separability naturally leads one 
to the countable case. 

In 
\cite{Ishiki2023Ury}, 
the author 
introduced 
\emph{petaloid spaces}, 
which play the role of 
$R$-Urysohn
universal ultrametric spaces
when
$R$
is uncountable. 

We begin by fixing notation. 
A subset
$E$
of
$[0, \infty)$
is said to be
 \emph{semi-sporadic} if 
there exists a strictly decreasing sequence 
$\{a_{i}\}_{i\in \zz_{\ge 0}}$
in
$(0, \infty)$
such that 
$\lim_{i\to \infty}a_{i}=0$
and 
$E=\{0\}\cup \{\, a_{i}\mid i\in \zz_{\ge 0}\, \}$.
A subset of
$[0, \infty)$
is 
said to be
\emph{tenuous} 
if it is finite or semi-sporadic
(see \cite{Ishiki2023const}). 
For a range set
$R$,
we denote by
$\yoofam{R}$
the set of all tenuous 
range subsets of
$R$.
Let us recall the definition of petaloid spaces
from 
\cite{Ishiki2023Ury}. 
%dfbegin
\begin{df}\label{df:petal}
Let 
$R$
be a  range set.  
We say that a metric space  
$(X, d)$
is \emph{$R$-petaloid} 
if 
it is an 
$R$-valued
ultrametric space 
and 
there exists a family 
$\{\yopetal{X}{S}\}_{S\in \yoofam{R}}$
of subspaces of 
$X$
satisfying the following properties:
%enubegin
\begin{enumerate}[label=\textup{(P\arabic*)}]

\item\label{item:pr:sep}
For every  
$S\in \yoofam{R}$,
the subspace 
 $(\yopetal{X}{S}, d)$
 is isometric to 
 the 
 $S$-Urysohn
universal ultrametric space.

\item\label{item:pr:cup}
We have 
$\bigcup_{S\in \yoofam{R}}\yopetal{X}{S}=X$.

\item\label{item:pr:cap}
If 
$S,  T\in \yoofam{R}$,
then
$\yopetal{X}{S}\cap \yopetal{X}{T}=\yopetal{X}{S\cap T}$.

\item\label{item:pr:distance}
If
 $S, T\in \yoofam{R}$
 and
  $x\in \yopetal{X}{T}$,
  then
$d(x, \yopetal{X}{S})$
belongs to 
$(T\setminus S)\cup \{0\}$.

\end{enumerate}
%enuend
We call the family 
\[
\{\yopetal{X}{S}\}_{S\in \yoofam{R}}
\]
an 
\emph{$R$-petal
structure on
$X$},
and call
$\yopetal{X}{S}$
the 
$S$-piece
of the 
$R$-petal
structure 
above. 
\end{df}
%dfend
Observe that even when 
$R$
is countable, 
the 
$R$-Urysohn
universal ultrametric space
has an 
$R$-petal
structure satisfying 
the conditions 
\ref{item:pr:sep}--\ref{item:pr:distance}
(see \cite{Ishiki2023Ury}).
This means that 
petaloid spaces are 
natural generalizations of 
Urysohn universal ultrametric spaces. 

The following theorem is taken from 
\cite[Theorem 2.3]{Ishiki2023halo}
(see also 
\cite{Ishiki2023Ury} 
for petaloid spaces and 
\cite[Propositions 20.2 and 21.1]{MR2444734}). 
In the cited papers, 
the existence and uniqueness of petaloid spaces are stated for 
uncountable range sets. 
The same proofs apply without change to countable range sets; 
in that case, the resulting petaloid space is the separable 
$R$-Urysohn
universal ultrametric space. 
\begin{thm}\label{thm:18:petapeta}
Let
$R$
be a range set. 
The following statements hold:
\begin{enumerate}[label=\textup{(\arabic*)}]
\item\label{item:uni}
There exists a unique
$R$-petaloid
ultrametric space
up to isometry. 
\item 
The
$R$-petaloid
ultrametric space is complete. 
Moreover, it is 
$\youfin(R)$-injective.
\item 
the range set
$R$
 is uncountable if and only if 
the
$R$-petaloid
ultrametric space is
non-separable. 
\item 
Every separable
$R$-valued
ultrametric space 
can be isometrically embedded into 
the
$R$-petaloid
ultrametric space. 
\end{enumerate}
\end{thm}

Based on 
Theorem 
\ref{thm:18:petapeta}, 
for a range set
$R$,
we denote by 
$(\yowusp{R}, \yowudis{R})$
the 
$R$-Urysohn
universal 
ultrametric space if
$R$
is countable; 
otherwise, 
the
$R$-petaloid
ultrametric space. 
In what follows, 
by abuse of notation, 
we call 
$(\yowusp{R}, \yowudis{R})$
the 
\emph{$R$-Urysohn
universal ultrametric space}
even if
$R$
is uncountable.

\subsection{A coordinate model for petaloid spaces}

We describe 
an example of a petaloid space. 
Let 
$N$
be 
a countably infinite set containing 
a distinguished point 
$0$.
Let 
$R$
be a range set. 
We denote by
 $\yomapsco{R}{N}$
the set of all functions 
$f\colon R\to N$
 such that
  $f(0)=0$
   and 
the set 
$\{0\}\cup \{\, x\in R \mid f(x)\neq 0\, \}$
is tenuous. 
For 
$f, g\in \yomapsco{R}{N}$,
we define an 
$R$-ultrametric
$\yomaindisco$
 on 
$\yomapsco{R}{N}$
by 
\[
\yomaindisco(f, g)=
\max\{\, r\in R\mid f(r)\neq g(r)\, \}
\]
 if
  $f\neq g$;
otherwise, 
$\yomaindisco(f, g)=0$.
This maximum exists because the set 
\[
\{0\}\cup \{\, r\in R\mid f(r)\neq g(r)\,\}
\]
is tenuous. 
For more information on this
construction, we refer the reader to 
\cite{MR2754373} and 
\cite{MR2667917}. 
%%lembe
\begin{lem}\label{lem:Gisom}
For every 
 countably infinite set 
$N$
containing 
$0$
 and 
every range set 
$R$,
the ultrametric space 
$(\yomapsco{R}{N}, \yomaindisco)$
 is 
isometric to 
$(\yowusp{R}, \yowudis{R})$.
\end{lem}
%%lemed
%proofbe
\begin{proof}
The lemma follows from 
\cite[Theorem 1.3]{Ishiki2023Ury}. 
Although 
in \cite{Ishiki2023Ury}, 
it is only shown that 
$(\yomapsco{R}{\zz_{\ge0}}, \yomaindisco)$
is 
isometric to 
$(\yowusp{R}, \yowudis{R})$,
the proof in 
\cite{Ishiki2023Ury}
 is still valid 
for an arbitrary countably infinite set 
$N$.
\end{proof}
%%proofed

%lembe 
\begin{lem}\label{lem:antitone}
Let 
$A$
be a 
subset of 
$\rr$,
and 
$\yoratio\in (1, \infty)$.
Put 
$S=\{0\}\sqcup \{\, \yoratio^{-g}\mid g\in A\, \}$.
Then the following statements are equivalent: 
\begin{enumerate}[label=\textup{(\arabic*)}]
\item 
for every 
$n\in \zz_{\ge 0}$,
the set 
$A\cap (-\infty, n]$
is finite; 
\item 
the set 
$S$
is 
tenuous. 
\end{enumerate}
\end{lem}
%lemed
%proofbe
\begin{proof}
The lemma 
follows from 
\cite[Lemma 2.12]{Ishiki2023const}
and the fact that 
a map 
$f\colon \rr\to \rr$
defined by 
 $f(x)=\yoratio^{-x}$
 reverses the 
 order on 
 $\rr$.
\end{proof}
%proofed

\subsection{Levi--Civita realizations}

We use the complete system of representatives 
$\yorep$
fixed in Definition 
\ref{df:tau}. 
For 
$x\in \youulc{G}{\yosmf{k}}{q}{p}$,
we write 
$\yostan{G}{\yosmf{k}}{\yorep}(x)$
for its standard representation with coefficients in 
$\yorep$.
When 
$q=p$,
we put 
$\yostan{G}{\yosmf{k}}{\yorep}(x)=x$.
When 
$q\neq p$,
this denotes the unique member of 
$\yostan{G}{\yosmf{k}}{\yorep}(\yopd{G}{\yosmf{k}})$
whose image under 
$\yoproj$
is 
$x$;
its existence and uniqueness follow from Corollary 
\ref{cor:216}.

%lembe
\begin{lem}\label{lem:cond1isom}
Let 
$\yoratio\in (1, \infty)$,
$G\in \yogset$,
$(q, p)\in \yoqpset$,
$\yosmf{k}$
be a perfect field of cardinality 
$\card(\yosmf{k})=\aleph_{0}$
and 
characteristic 
$p$.
Put 
\[
R=\{0\}\cup \{\, \yoratio^{-g}\mid g\in G\, \}.
\] 
Use the complete system of representatives 
$\yorep$
fixed in Definition 
\ref{df:tau}. 
Take 
$S\in \yoofam{R}$.
Define 
a
subset 
$\yopetal{\youulc{G}{\yosmf{k}}{q}{p}}{S}$
of 
$\youulc{G}{\yosmf{k}}{q}{p}$
by 
the set of all 
$f\in \youulc{G}{\yosmf{k}}{q}{p}$
such that 
\[
\{\, \yoratio^{-g}\mid g\in \supp(\yostan{G}{\yosmf{k}}{\yorep}(f))\, \}\yosub S,
\] 
and 
define 
an ultrametric 
$d$
on
$\youulc{G}{\yosmf{k}}{q}{p}$
 by 
\[
d(x, y)=  \yonabs{\youhavq{G}{\yosmf{k}}{q}{p}}{\yoratio}{x-y}. 
\]
Then 
$(\yopetal{\youulc{G}{\yosmf{k}}{q}{p}}{S}, d)$
is isometric to 
$(\yowusp{S}, \yowudis{S})$.
\end{lem}
%lemed
%proofbe
\begin{proof}
Since 
$\card(\yosmf{k})=\aleph_{0}$,
we also have 
$\card(\yorep)=\aleph_{0}$.
For every
$x\in\youulc{G}{\yosmf{k}}{q}{p}$,
there is a unique family
$\{s_{x}(g)\}_{g\in G}$
in
$\yorep$
such that
\[
\yostan{G}{\yosmf{k}}{\yorep}(x)
=\sum_{g\in G}s_{x}(g)t^{g}.
\]
Indeed, this is the usual coefficient expansion when
$q=p$,
and it follows
from Corollary \ref{cor:216} when
$q\neq p$.
Moreover, the set
\[
A_{x}=\{\,g\in G\mid s_{x}(g)\neq0\,\}
\]
satisfies the condition \ref{item:19finc}.

Take
$x\in\yopetal{\youulc{G}{\yosmf{k}}{q}{p}}{S}$.
We define a map
$T(x)\colon S\to \yorep$
by 
$T(x)(0)=0$
and, for
$r\neq 0$,
$T(x)(r)=s_{x}(-\log_{\yoratio}(r))$.
This is well-defined because
$S\setminus\{0\}\subset\{\,\yoratio^{-g}\mid g\in G\,\}$.
Since
$x$
belongs to the
$S$-piece,
we have
\[
\{0\}\cup\{\,r\in S\mid T(x)(r)\neq0\,\}
=\{0\}\cup\{\,\yoratio^{-g}\mid g\in A_{x}\,\}
\subset S.
\]
The set on the right-hand side is tenuous by
Lemma \ref{lem:antitone}.
Thus
$T(x)\in \yomapsco{S}{\yorep}$
for all 
$x\in \yopetal{\youulc{G}{\yosmf{k}}{q}{p}}{S}$.
We therefore obtain a map
\[
T\colon \yopetal{\youulc{G}{\yosmf{k}}{q}{p}}{S}
\to \yomapsco{S}{\yorep}
\]
defined by
$x\mapsto T(x)$.

We first prove that
$T$
is injective.
Assume that
$T(x)=T(y)$.
Then 
for every
$g\in G$
such that
$\yoratio^{-g}\in S$,
we have
$s_{x}(g)=s_{y}(g)$. 
Thus
$s_{x}(g)=s_{y}(g)$
for every
$g\in G$.
The uniqueness of the standard representation yields
$x=y$.

To prove surjectivity, take
$h\in\yomapsco{S}{\yorep}$
and define
$c(g)\in\yorep$
for
$g\in G$
by
\[
c(g)=
\begin{cases}
h(\yoratio^{-g}) & \text{if
$\yoratio^{-g}\in S$;}\\
0 & \text{otherwise.}
\end{cases}
\]
Put
\[
A=\{\,g\in G\mid c(g)\neq0\,\}.
\]
Since
$h(0)=0$,
we have
\[
\{0\}\cup\{\,\yoratio^{-g}\mid g\in A\,\}
=\{0\}\cup\{\,r\in S\mid h(r)\neq0\,\}.
\]
The latter set is tenuous by the definition of
$\yomapsco{S}{\yorep}$.
Hence Lemma \ref{lem:antitone} shows that
$A$
satisfies
\ref{item:19finc}.
If
$q=p$,
the series
\[
x=\sum_{g\in G}c(g)t^{g}
\]
therefore belongs to
$\yolc{G}{\yosmf{k}}$.
If
$q\neq p$,
the same series belongs to
$\yostan{G}{\yosmf{k}}{\yorep}(\yopd{G}{\yosmf{k}})$,
and Corollary \ref{cor:216} gives an element
$x\in\yoplc{G}{\yosmf{k}}{p}$
having this standard representation.
In either case, the support condition above shows that
$x\in\yopetal{\youulc{G}{\yosmf{k}}{q}{p}}{S}$,
and the definition of
$c(g)$
gives
$T(x)=h$.
Thus
$T$
is surjective.

It remains to verify that
$T$
preserves distances.
Let
$x,y$
belong to the
$S$-piece.
The assertion is obvious  if
$x=y$. 
We may  assume that
$x\neq y$. 
Put 
\[
m=\min\{\,g\in G\mid s_{x}(g)\neq s_{y}(g)\,\}.
\]
This minimum exists because both coefficient supports satisfy
\ref{item:19finc}.
When
$q=p$,
the definition of the Hahn valuation gives
\[
\youhavq{G}{\yosmf{k}}{q}{p}(x-y)=m.
\]
When
$q\neq p$,
the coefficients
$s_{x}(m)$
and
$s_{y}(m)$
are distinct
representatives of residue classes.
Hence their difference has valuation
$0$
in
$\yowitto{\yosmf{k}}$,
and the coefficients of smaller exponent in
the two standard representations agree.
It follows again that
\[
\youhavq{G}{\yosmf{k}}{q}{p}(x-y)=m.
\]
Consequently, we have 
$d(x,y)=\yoratio^{-m}$.
On the other hand, the points at which 
$T(x)$
and 
$T(y)$
 differ are exactly
$\{\,\yoratio^{-g}\mid s_{x}(g)\neq s_{y}(g)\,\}$.
Since the map
 $g\mapsto\yoratio^{-g}$
 reverses the order, its largest
element is 
$\yoratio^{-m}$.
Therefore we observe that 
\[
\yomaindisco(T(x),T(y))=\yoratio^{-m}=d(x,y).
\]
Thus 
$T$
is an isometric bijection from the 
$S$-piece
onto
$(\yomapsco{S}{\yorep},\yomaindisco)$.
Finally, since
$\card(\yorep)=\aleph_{0}$,
Lemma \ref{lem:Gisom} completes the proof.
\end{proof}
%proofed

%statebody
%thmbegin
\begin{thm}\label{thm:mainury}
Let 
$\yoratio\in (1, \infty)$,
$G\in \yogset$,
$(q, p)\in \yoqpset$,
$\yosmf{k}$
be a perfect field of cardinality 
$\card(\yosmf{k})=\aleph_{0}$
and 
characteristic 
$p$.
Put 
$R=\{0\}\cup \{\, \yoratio^{-g}\mid g\in G\, \}$.
Define 
an ultrametric 
$d$
on
$\youulc{G}{\yosmf{k}}{q}{p}$
 by 
\[
d(x, y)=\yonabs{\youhavq{G}{\yosmf{k}}{q}{p}}{\yoratio}{x-y}. 
\]
Then 
the 
Levi--Civita field
$(\youulc{G}{\yosmf{k}}{q}{p}, d)$
is 
isometric to 
$(\yowusp{R}, \yowudis{R})$.
\end{thm}
%thmend
%proof
\begin{proof}
Use the complete system of representatives 
$\yorep$
fixed in Definition 
\ref{df:tau}. 
We define 
$\yopetal{\youulc{G}{\yosmf{k}}{q}{p}}{S}$
as in 
Lemma 
\ref{lem:cond1isom}. 
Let us 
 prove
that 
$\{\yopetal{\youulc{G}{\yosmf{k}}{q}{p}}{S}\}_{S\in \yoofam{R}}$
satisfies 
the conditions 
\ref{item:pr:sep}--\ref{item:pr:distance}. 

Lemma 
\ref{lem:cond1isom}
implies 
the condition 
\ref{item:pr:sep}.
We next verify \ref{item:pr:cap}.
Let
$S,T\in\yoofam{R}$.
The set
$S\cap T$
is also a tenuous range subset of
$R$.
For
$f\in\youulc{G}{\yosmf{k}}{q}{p}$,
put
$U_{f}=\{\,\yoratio^{-g}\mid
g\in\supp(\yostan{G}{\yosmf{k}}{\yorep}(f))\,\}$.
Then
$U_{f}\subset S\ \text{and}\ U_{f}\subset T
$
if and only if
$
U_{f}\subset S\cap T$.
Therefore
\[
\yopetal{\youulc{G}{\yosmf{k}}{q}{p}}{S}
\cap
\yopetal{\youulc{G}{\yosmf{k}}{q}{p}}{T}
=
\yopetal{\youulc{G}{\yosmf{k}}{q}{p}}{S\cap T},
\]
which proves \ref{item:pr:cap}.

Next we show 
\ref{item:pr:cup}. 
Take 
$f\in \youulc{G}{\yosmf{k}}{q}{p}$
and put
\[
A=\supp(\yostan{G}{\yosmf{k}}{\yorep}(f)).
\]
The set
$A$
satisfies \ref{item:19finc}: this follows from the definition
of the ordinary Levi--Civita field when
$q=p$,
and from
Corollary \ref{cor:216} when
$q\neq p$.
Define
\[
S=\{0\}\cup \{\, \yoratio^{-g}\mid g\in A\, \}.
\]
Since
$A\subset G$,
we have
$S\subset R$,
and
Lemma \ref{lem:antitone} gives
$S\in\yoofam{R}$.
Hence 
$f\in \yopetal{\youulc{G}{\yosmf{k}}{q}{p}}{S}$.
Thus 
\ref{item:pr:cup} is true. 

Now we show 
\ref{item:pr:distance}. 
Let
$S,T\in \yoofam{R}$
and 
$x\in \yopetal{\youulc{G}{\yosmf{k}}{q}{p}}{T}$.
If 
$x\in \yopetal{\youulc{G}{\yosmf{k}}{q}{p}}{S}$,
then 
$d(x,\yopetal{\youulc{G}{\yosmf{k}}{q}{p}}{S})=0$.
Thus 
\ref{item:pr:distance} holds in this case. 
We may assume that 
$x\not \in \yopetal{\youulc{G}{\yosmf{k}}{q}{p}}{S}$.
Put 
\[
U=\{\, \yoratio^{-g}\mid 
g\in \supp(\yostan{G}{\yosmf{k}}{\yorep}(x))\, \}.
\] 
Since 
$x\not \in \yopetal{\youulc{G}{\yosmf{k}}{q}{p}}{S}$,
we have 
$U\setminus S\neq \emptyset$.
Since 
$U\setminus S$
is a non-empty subset of the tenuous set 
$T$,
it has a maximum. 
Put  
$u=\max (U\setminus S)$
and 
$m=-\log_{\yoratio}(u)$.
Notice that 
$\yocoef{x}{m}\neq 0$
(see Definition \ref{df:tau}). 
Define a point 
$y\in \youulc{G}{\yosmf{k}}{q}{p}$
by 
\[
y=\sum_{g\in G\cap (-\infty, m)}
\yocoef{x}{g}\youit^{g}. 
\]
The support of the displayed expansion is a subset of the support of
$x$
and hence satisfies \ref{item:19finc}.
Thus it defines an element of
$\yolc{G}{\yosmf{k}}$
when
$q=p$.
When
$q\neq p$,
Corollary \ref{cor:216} shows that it is the standard
representation of an element of
$\yoplc{G}{\yosmf{k}}{p}$.
Consequently, in either case,
$y\in\youulc{G}{\yosmf{k}}{q}{p}$.
Moreover, if
$g<m$
and
$\yocoef{x}{g}\neq0$,
then 
$\yoratio^{-g}>u$;
the maximality of
$u$
in
$U\setminus S$
therefore gives 
$\yoratio^{-g}\in S$.
Thus
$y\in \yopetal{\youulc{G}{\yosmf{k}}{q}{p}}{S}$.
The standard coefficients of
$x$
and
$y$
agree below
$m$
and differ at
$m$,
so the coefficient calculation in the proof of
Lemma \ref{lem:cond1isom} gives
$d(x,y)=\yoratio^{-m}=u$.

Let
$z\in\yopetal{\youulc{G}{\yosmf{k}}{q}{p}}{S}$
be arbitrary.
Since
$u=\yoratio^{-m}\notin S$,
the coefficient of
$z$
at
$m$
is zero,
whereas
$\yocoef{x}{m}\neq0$.
If
\[
n=\min\{\,g\in G\mid\yocoef{x}{g}\neq\yocoef{z}{g}\,\},
\]
then
$n\le m$.
The same coefficient calculation gives
\[
d(x,z)=\yoratio^{-n}\ge\yoratio^{-m}=u.
\]
Since
$z$
was arbitrary, and since the point
$y$
above realizes the value
$u$,
we obtain
\[
d(x, \yopetal{\youulc{G}{\yosmf{k}}{q}{p}}{S})=u\in T\setminus S, 
\]
which means that 
the condition 
\ref{item:pr:distance}
holds. 

We have proved that the family of
$S$-pieces
satisfies
\ref{item:pr:sep}--\ref{item:pr:distance}.
If
$R$
is uncountable, this makes
$(\youulc{G}{\yosmf{k}}{q}{p},d)$
an
$R$-petaloid
ultrametric space.
The uniqueness assertion in Theorem \ref{thm:18:petapeta} then gives an
isometry with
$(\yowusp{R},\yowudis{R})$.
\end{proof}
%proofend

\section{Universality and valued fields}\label{sec:univvalued}

We give another 
application of 
the theory of 
Urysohn universal 
ultrametric spaces
to valued fields. 
For a 
class 
$\yochara{C}$
of metric spaces, 
we say that 
a metric space 
$(X, d)$
is 
\emph{$\yochara{C}$-universal} 
or 
\emph{universal for
$\yochara{C}$}
if for every
$(A, e)\in \yochara{C}$
there exists an isometric embedding 
$f\colon A\to X$.
An ultrametric space
$(X, d)$
is 
said to be 
\yoemph{$(R, \aleph_{0})$-haloed} if 
for every
$a\in X$
and for every
$r\in R\setminus \{0\}$,
there exists a subset
 $A$
of 
$\yocball(a, r)$
such that 
$\aleph_{0}\le \card(A)$
and 
$d(x, y)=r$
for all distinct 
$x, y\in A$
(see \cite{Ishiki2023halo}). 

The following characterization will be used below.
\begin{lem}\label{lem:haloeduniversal}
Let
$R$
be a range set and let
$(X,d)$
be an
$R$-valued
ultrametric space. 
Then
$(X,d)$
is
$(R,\aleph_{0})$-haloed
if and only if it is 
$\youfin(R)$-injective.
\end{lem}
\begin{proof}
This is \cite[Theorem 1.1]{Ishiki2023halo}.
\end{proof}

%thmbegin 
\begin{thm}\label{thm:fin}
Let 
$\yoratio\in (1, \infty)$,
$G\in \yogset$,
and 
let 
$(K, v)$
be a complete valued field
such that 
$G\yosub v(K)$,
and
$\yorcf{K}{v}$
is an infinite set. 
Put 
$R=\{0\}\cup \{\, \yoratio^{-g}\mid g\in G\, \}$.
Then 
$(K, \yonabs{v}{\yoratio}{*})$
is
$(R,\aleph_{0})$-haloed.
In particular, it is universal for all 
separable
$R$-valued
 ultrametric spaces. 
\end{thm}
%rhmend
%proobe 
\begin{proof}
Since 
$\aleph_{0}\le \card(\yorcf{K}{v})$,
for every 
$a\in K$
and every 
$g\in G$,
we can choose 
$b\in K$
with 
$v(b)=g$,
and the infinitude of 
$\yorcf{K}{v}$
gives a countably infinite set 
$C\yosub \yovring{K}{v}$
whose elements have pairwise distinct residue classes. 
Then 
$a+bC$
contains a countably infinite subset of 
$\yocball(a, \yoratio^{-g})$
whose distinct points are at distance 
$\yoratio^{-g}$.
Hence 
$(K, \yonabs{v}{\yoratio}{*})$
is 
$(R, \aleph_{0})$-haloed.
Since it is complete, \cite[Theorem 1.5]{Ishiki2023halo} shows that it 
contains an isometric copy of
$(\yowusp{R},\yowudis{R})$.
The latter space is universal for all separable
$R$-valued
ultrametric spaces, 
which proves the final assertion. 
\end{proof}
%proofed

\begin{thm}\label{thm:separablevaluedurysohn}
Let
$\yoratio\in(1,\infty)$,
$G\in\yogset$,
and let
$(K,v)$
be a 
separable complete valued field. 
Assume that
$v(K\setminus\{0\})=G$
and that the residue field 
$\yorcf{K}{v}$
 is infinite. 
Put
$R=\{0\}\cup\{\,\yoratio^{-g}\mid g\in G\,\}$.
Then the ultrametric space
$(K,\yonabs{v}{\yoratio}{*})$
is isometric to the 
$R$-Urysohn
universal ultrametric space
$(\yowusp{R},\yowudis{R})$.
\end{thm}
\begin{proof}
Let
$D$
be a countable dense subset of 
$(K,\yonabs{v}{\yoratio}{*})$.
For every distinct
$x,y\in K$,
we can choose
$a,b\in D$
such that
\[
\yonabs{v}{\yoratio}{x-a}<\yonabs{v}{\yoratio}{x-y}
\quad\text{and}\quad
\yonabs{v}{\yoratio}{y-b}<\yonabs{v}{\yoratio}{x-y}.
\]
The ultrametric inequality then gives
\[
\yonabs{v}{\yoratio}{a-b}=\yonabs{v}{\yoratio}{x-y}.
\]
Thus every value in
$G=v(K\setminus\{0\})$
is determined by the distance 
between two points of
$D$,
and hence
$G$
is countable. 
The residue field is also countable. 
Indeed, representatives of distinct residue classes in the valuation ring 
are pairwise at distance
$1$.
A separable metric space contains no uncountable set whose distinct points 
are pairwise at distance
$1$.
Thus
$\yorcf{K}{v}$
is countably infinite. 

Theorem \ref{thm:fin} shows that 
$(K,\yonabs{v}{\yoratio}{*})$
is
$(R,\aleph_{0})$-haloed.
Hence it is
$\youfin(R)$-injective
by 
Lemma \ref{lem:haloeduniversal}. 
Since
$G$
is countable, so is
$R$.
Consequently,
$(K,\yonabs{v}{\yoratio}{*})$
is a separable complete 
$\youfin(R)$-injective
$R$-valued
ultrametric space. 
The uniqueness of the
$R$-Urysohn
universal ultrametric space 
therefore gives the desired isometry.
\end{proof}

We record several standard applications of 
Theorem \ref{thm:separablevaluedurysohn}. 
%corbegin
\begin{cor}\label{cor:padicc}
Let
$\yoratio\in(1,\infty)$.
Put
\[
\yorsub{\zz}=\{0\}\cup\{\,\yoratio^{-n}\mid n\in\zz\,\}
\quad\text{and}\quad
\yorsub{\qq}=\{0\}\cup\{\,\yoratio^{-g}\mid g\in\qq\,\}.
\]
The following valued fields are isometric to the indicated 
Urysohn universal ultrametric spaces.
\begin{enumerate}[label=\textup{(\arabic*)}]

\item
Let
$p$
be a prime.
Then the completion of the maximal unramified extension of
$\yopnbf{p}$,
the field 
$(\yopnburf{p},\yonabs{\yopnbv{p}}{\yoratio}{*})$,
is isometric to
$(\yowusp{\yorsub{\zz}},\yowudis{\yorsub{\zz}})$.

\item
Let
$p$
be a prime.
Then the field of
$p$-adic
complex numbers
$(\yopnbcf{p},\yonabs{\yopnbv{p}}{\yoratio}{*})$
is isometric to
$(\yowusp{\yorsub{\qq}},\yowudis{\yorsub{\qq}})$.

\item
Let
$k$
be a countably infinite field, equip
$\yolaurentf{k}$
with the 
$t$-adic
valuation
$v_{t}$.
Then
$(\yolaurentf{k},\yonabs{v_{t}}{\yoratio}{*})$
is isometric to
$(\yowusp{\yorsub{\zz}},\yowudis{\yorsub{\zz}})$.

\item
Let
$k$
be a countably infinite field, and fix an extension of
$v_{t}$
to
an algebraic closure
$\overline{\yolaurentf{k}}$.
Then
$(\yolaurentcf{k},\yonabs{v_{t}}{\yoratio}{*})$
is isometric to
$(\yowusp{\yorsub{\qq}},\yowudis{\yorsub{\qq}})$.

\end{enumerate}
\end{cor}
%corend
\begin{proof}
In \textup{(1)}, the maximal unramified extension is a countable union of
finite extensions of
$\yopnbf{p}$.
Hence
$\yopnburf{p}$
is separable
and complete.
The maximal unramified extension has value group
$\zz$
and its residue field is
$\overline{\yogf{p}}$;
see
\cite[Chapter III, Section 5 and Chapter IV, Section 4]{MR0554237}.
Since completion preserves both the value group and the residue field
\cite[Chapter 2, Section 2, item (M)]{MR1677964},
the same holds for
$\yopnburf{p}$.

In \textup{(2)}, the field
$\yopnbcf{p}$
is separable and complete, with
value group
$\qq$
and residue field
$\overline{\yogf{p}}$. 

In \textup{(3)}, the field
$\yolaurentf{k}$
is complete, the countable set
$k[t,t^{-1}]$
is dense in it, its value group is
$\zz$,
and its residue
field is
$k$.

In \textup{(4)}, 
take an algebraic closure 
$\overline{k(t)}$
of 
$k(t)$,
endow it with an extension of the 
$t$-adic
valuation, 
and let 
$C$
be its
completion.
By 
\cite[Theorem 17.1]{MR2444734}, 
the field 
$C$
is algebraically closed.
The (topological) closure of 
$k(t)$
in 
$C$
is canonically isomorphic to
$\yolaurentf{k}$.
We may therefore take 
$\overline{\yolaurentf{k}}$
 to be the algebraic
closure of 
$\yolaurentf{k}$
in 
$C$.
Since every element of 
$\overline{k(t)}$
 is algebraic over
$\yolaurentf{k}$,
we have
$\overline{k(t)}\subset \overline{\yolaurentf{k}}\subset C$.
Since 
$\overline{k(t)}$
is dense in 
$C$,
it is also dense in
$\overline{\yolaurentf{k}}$.
The algebraic closure 
$\overline{\yolaurentf{k}}$
 has value group 
 $\qq$
and residue field 
$\overline{k}$,
and its completion is an immediate
extension.
Thus 
$\yolaurentcf{k}$
 is separable and complete, with value group 
 $\qq$
and residue field 
$\overline{k}$.

In every case the residue field is infinite.
Therefore all four assertions
follow from Theorem \ref{thm:separablevaluedurysohn}.
\end{proof}

\begin{lem}\label{lem:diamtenuous}
Assume that every strictly decreasing sequence in
$R\setminus\{0\}$
converges to
 $0$.
 Let
 $\{B_{n}\}_{n\in\zz_{\ge0}}$
 be a
sequence of non-empty closed or open balls in
$(\yomapsco{R}{\zz_{\ge0}},\yomaindisco)$
such that 
$B_{n+1}\yosub B_{n}$
for every 
$n\in \zz_{\ge 0}$.
Then 
there
exists 
$f\in \yomapsco{R}{\zz_{\ge0}}$
such that 
$f\in \bigcap_{n\in \zz_{\ge 0}}B_{n}$.  
\end{lem}
\begin{proof}
If the sequence
$\{\yodiam(B_{n})\}$
is eventually constant, then the
description of balls in the coordinate model shows that a proper inclusion
at unchanged diameter can occur only when a closed ball is replaced by the
corresponding open ball.
Hence the balls are eventually equal, and their
intersection is non-empty.
Thus we may assume that
$\{\yodiam(B_{n})\}$
is not eventually constant.
Choose a strictly increasing sequence 
$\{n_{j}\}_{j\in\zz_{\ge0}}$
such that 
$\yodiam(B_{n_{j+1}})<\yodiam(B_{n_{j}})$
for every 
$j\in\zz_{\ge0}$.
By the definition of the diameter, 
we can take 
$x_{j},y_{j}\in B_{n_{j}}$
such that 
$\yodiam(B_{n_{j+1}})
<\yomaindisco(x_{j},y_{j})
\le \yodiam(B_{n_{j}})$.
Put 
$r_{j}=\yomaindisco(x_{j},y_{j})$.
Then 
$r_{j}\in R\setminus\{0\}$
and 
$r_{j+1}
\le \yodiam(B_{n_{j+1}})
<r_{j}$.
The assumption on 
$R$
 gives $\lim_{j\to\infty}r_{j}=0$.
Thus
we obtain 
$\lim_{n\to \infty}\yodiam(B_{n})=0$.
For each 
$n\in \zz_{\ge 0}$
take a center 
$a_{n}$
of 
$B_{n}$.
For
$s\in R\setminus\{0\}$,
put
$N(s)=\min\{\,n\in\zz_{\ge0}\mid \yodiam(B_{n})<s\,\}$. 
Define
$f\colon R\to\zz_{\ge0}$
by 
$f(s)=0$ if 
$s=0$; 
otherwise, 
$f(s)=a_{N(s)}(s)$. 
From 
 the assumption on 
$R$, 
it follows that 
for every
$\epsilon>0$,
the set 
$\{\,s\in R\mid \epsilon\le s,\ f(s)\neq0\,\}$
is finite. 
Thus 
$\{0\}\cup\{\,s\in R\mid f(s)\neq0\,\}$
is tenuous and 
$f\in \yomapsco{R}{\zz_{\ge0}}$. 
By the definition of 
$f$,  
and by the 
fact that 
$a_{n}$
 is  a center of 
$B_{n}$, 
we conclude that
$f\in B_{n}$ 
 for every 
$n\in \zz_{\ge 0}$. 
\end{proof}

\begin{prop}\label{prop:sphericalcriterion}
The space
$(\yowusp{R},\yowudis{R})$
is spherically complete
if and only if every strictly decreasing sequence in
$R\setminus\{0\}$
converges to
$0$.
\end{prop}
\begin{proof}
By Lemma \ref{lem:Gisom}, identify
$(\yowusp{R},\yowudis{R})$
with
$(\yomapsco{R}{\zz_{\ge0}},\yomaindisco)$.

Suppose first that a strictly decreasing sequence
$\{r_{n}\}_{n\in\zz_{\ge0}}$
in
$R\setminus\{0\}$
does not converge to
$0$.
Its limit is positive.
Define
$f_{n}(r_{i})=1$
for
$i<n$
and put
$f_{n}=0$
elsewhere.
The closed balls
$\{\,\yocball(f_{n},r_{n})\mid n\in\zz_{\ge0}\,\}$
form a decreasing sequence.
Any point in their intersection would satisfy
$f(r_{n})=1$
for every
$n$,
contradicting the tenuous-support condition
because the
$r_{n}$
have a positive limit.

Conversely, assume that every strictly decreasing sequence in
$R\setminus\{0\}$
converges to
$0$,
and let
$\{B_{n}\}$
be a decreasing
sequence of non-empty closed or open balls.
By Lemma
\ref{lem:diamtenuous},
there exists
$f\in\yomapsco{R}{\zz_{\ge0}}$
such that
$f\in B_{n}$
for every
$n\in\zz_{\ge0}$.
Hence
$\bigcap_{n\in \zz_{\ge 0}} B_{n}\neq\emptyset$.
\end{proof}

\begin{cor}\label{cor:notspherical}
Let
$\yoratio\in(1,\infty)$,
$G\in\yogset$,
$(q,p)\in\yoqpset$,
and let
$\yosmf{k}$
be a perfect field of characteristic
$p$
such that
$\card(\yosmf{k})=\aleph_{0}$.
Put
\[
R=\{0\}\cup\{\,\yoratio^{-g}\mid g\in G\,\}.
\]
Define an ultrametric
$d$
on
$\youhaq{G}{\yosmf{k}}{q}{p}$
by
\[
d(x,y)=
\yonabs{\youhavq{G}{\yosmf{k}}{q}{p}}{\yoratio}{x-y}.
\]
Then the following conditions are equivalent:
\begin{enumerate}[label=\textup{(\arabic*)}]
\item
$G$
is order-isomorphic to
$\zz$;
\item
$(\youhaq{G}{\yosmf{k}}{q}{p},d)$
is isometric to the
$R$-Urysohn
universal ultrametric space.
\end{enumerate}
\end{cor}
\begin{proof}
Assume first that
$G$
is order-isomorphic to
$\zz$.
Then
$G=c\zz$
for some
$c>0$.
Every well-ordered subset
$A$
of
$G$
has a minimum, and the discreteness of
$G$
implies that
$A\cap(-\infty,n]$
is finite for every
$n\in\zz$.
Thus every Hahn series with exponents in
$G$
satisfies the condition
\ref{item:19finc}.
When
$q=p$,
this gives
$\yoha{G}{\yosmf{k}}=\yolc{G}{\yosmf{k}}$.
When
$q\neq p$,
it gives
$\yopd{G}{\yosmf{k}}
=\yoha{G}{\yowitto{\yosmf{k}}}$,
and hence
$\yopf{G}{\yosmf{k}}{p}=\yoplc{G}{\yosmf{k}}{p}$.
Therefore, in either case,
$\youhaq{G}{\yosmf{k}}{q}{p}=\youulc{G}{\yosmf{k}}{q}{p}$.
Condition \textup{(2)} now follows from Theorem \ref{thm:mainury}.

Conversely, assume that \textup{(2)} holds.
By Proposition \ref{prop:hahnspherical},
the Hahn field
$(\youhaq{G}{\yosmf{k}}{q}{p},d)$
is spherically complete.
Hence the
$R$-Urysohn
universal ultrametric space is also spherically complete.
Proposition \ref{prop:sphericalcriterion} shows that every strictly
decreasing sequence in
$R\setminus\{0\}$
converges to
$0$.
We claim that
$G$
is discrete in
$\rr$.
Otherwise, for every
$n\in \zz_{\ge 0}$
we could choose
$d_{n}\in G$
such that
$0<d_{n}<2^{-n}$.
Putting
$s_{n}=\sum_{i=0}^{n}d_{i}$,
we would obtain a strictly increasing bounded sequence
$\{s_{n}\}_{n\in \zz_{\ge 0}}$
in
$G$.
The sequence
$\{\yoratio^{-s_{n}}\}_{n\in \zz_{\ge 0}}$
would then be strictly decreasing in
$R\setminus\{0\}$
and would have a positive limit, a contradiction.
Thus
$G$
is discrete.
Every non-zero discrete subgroup of
$\rr$
has a least positive element
$c$
and is equal to
$c\zz$.
Consequently,
$G$
is order-isomorphic to
$\zz$.
\end{proof}

%%%%%%%%%%%%%%%%%%%%%%%%%%
%bibliography
%%%%%%%%%%%%%%%%%%%%%%%%%%

%READ!
%myplain is the same to amsplain style. 
%\nocite{*}
\bibliographystyle{myplaindoidoi}
\bibliography{../../../bibtex/UU.bib}
%%%%%%%%%%%%%%%%

%%%%%%%%%%%%%%%%%%%%%%%%%%%%%%

%%%%%%%%%%%%%%%%%%%%%%%%%%
%%%%%%%%%%%%%%%%%%%%%%%%%%
%%%%%%%%%%%%%%%%%%%%%%%%%%
%%%%%%%%%%%%%%%%%%%%%%%%%%
%%%%%%%%%%%%%%%%%%%%%%%%%%
%%%%%%%%%%%%%%%%%%%%%%%%%%
%%%%%%%%%%%%%%%%%%%%%%%%%%
%%%%%%%%%%%%%%%%%%%%%%%%%%
%%%%%%%%%%%%%%%%%%%%%%%%%%
%%%%%%%%%%%%%%%%%%%%%%%%%%
%%%%%%%%%%%%%%%%%%%%%%%%%%
%end
\end{document}